\documentclass[twoside]{amsart}

\numberwithin{equation}{section}
\usepackage{amsmath,amssymb,amsfonts,amsthm,latexsym}
\usepackage{amscd,graphicx,color,enumerate}
\usepackage[all]{xy}

\newtheorem{theorem}{Theorem}[section]
\newtheorem{lemma}[theorem]{Lemma}
\newtheorem{proposition}[theorem]{Proposition}
\newtheorem{corollary}[theorem]{Corollary}

\theoremstyle{definition}
\newtheorem{definition}[theorem]{Definition}

\theoremstyle{remark}
\newtheorem{remark}[theorem]{Remark}

\numberwithin{equation}{section}

\newcommand{\C}{\mathbb{C}}
\newcommand{\Sp}{\mathbb{S}}

\newcommand{\R}{\mathbb{R}}

\newcommand{\Z}{\mathbb{Z}}
\newcommand{\N}{\mathbb{N}}

\newcommand{\U}{\operatorname{U}}
\newcommand{\SU}{\operatorname{SU}}
\newcommand{\SL}{\operatorname{SL}}
\newcommand{\SO}{\operatorname{SO}}
\newcommand{\GL}{\operatorname{GL}}

\newcommand{\diag}{\operatorname{diag}}
\newcommand{\co}{\colon\thinspace}
\newcommand{\git}{/\!\!/}
\newcommand{\bs}{\boldsymbol}

\newcommand{\modu}{\operatorname{mod}}

\newcommand{\pr}{\mathrm{pr}}
\newcommand{\sm}{\mathrm{sm}}
\newcommand{\codim}{\operatorname{codim}}

\begin{document}

\title{When is a symplectic quotient an orbifold?}
\author{Hans-Christian Herbig}
\address{Charles University in Prague, Faculty of Mathematics and Physics,
Sokolovsk\'{a} 83, 186 75 Praha 8, Czech Republic}
\email{herbig@imf.au.dk}

\author{Gerald W. Schwarz}
\address{Department of Mathematics, Brandeis University,
Waltham, MA 02454-9110, USA}
\email{schwarz@brandeis.edu}

\author{Christopher Seaton}
\address{Department of Mathematics and Computer Science,
Rhodes College, 2000 N. Parkway, Memphis, TN 38112, USA}
\email{seatonc@rhodes.edu}

\thanks{The first author has been supported by the grant GA CR P201/12/G028.
The third author was supported by a Rhodes College Faculty Development Grant as well
as the E.C. Ellett Professorship in Mathematics.}
\begin{abstract}
Let $K$ be a compact
Lie group of positive dimension.  We show that for most unitary $K$-modules
the corresponding symplectic quotient is not
regularly symplectomorphic to a linear symplectic orbifold (the quotient of a unitary
module of a finite group).
When $K$ is connected, we show that even a symplectomorphism to a linear symplectic
orbifold does not exist.
Our results yield conditions that preclude the symplectic quotient of a
Hamiltonian $K$-manifold from being locally isomorphic to an orbifold.
As an application, we determine which unitary
$\SU_2$-modules yield symplectic quotients that
are $\Z$-graded regularly symplectomorphic to a linear symplectic orbifold.
We similarly determine which unitary circle representations yield symplectic
quotients that admit a regular diffeomorphism to a
linear symplectic orbifold.
\end{abstract}

\maketitle

\tableofcontents

% xxxxxxxxxxxxxxxxxxxxxxxxxxxxxxxxxxxxxxxxxxxxxxxxxxxxxxxxxxxxxxxxxxxxxxxxx
% xxxxxxxxxxxxxxxxxxxxxxxxxxxxxxxxxxxxxxxxxxxxxxxxxxxxxxxxxxxxxxxxxxxxxxxxx
% xxxxxxxxxxxxxxxxxxxxxxxxxxxxxxxxxxxxxxxxxxxxxxxxxxxxxxxxxxxxxxxxxxxxxxxxx

\section{Introduction}
\label{sec:Intro}

It has been observed that occasionally, symplectic quotients can be identified with orbifolds
(see e.g. \cite{GotayBos,LMS,FarHerSea}). In this paper we present results that indicate that
this situation is rather the exception than the rule.
We
mostly restrict our
attention to the case where the symplectic quotient, respectively finite quotient, comes from
a unitary representation on a hermitian, finite dimensional  vector space.

A symplectic quotient is equipped in a canonical way with an algebra of smooth functions, the
so-called \emph{smooth structure}. This algebra carries a canonical Poisson bracket.  In the
case of a unitary representation, this Poisson algebra contains the algebra of
\emph{regular functions} on the symplectic quotient as a Poisson subalgebra. The algebra of
regular functions can be described as an affine $\R$-algebra by means of invariant theory.
As we consider only reduction at the zero level of the moment map, the Poisson algebra of
regular functions is actually $\N$-graded, the Poisson bracket being of degree $-2$. The
algebra of regular functions can typically not be used to fully recover the symplectic
quotient, but rather its Zariski closure.

By the \emph{Lifting Theorem} \cite[Theorem 6]{FarHerSea}, a regular (Poisson)
map between the Zariski closures of symplectic quotients
can be lifted
in a unique way to a (Poisson) algebra morphism between the algebras of
smooth functions as long as it is compatible with the Hilbert embeddings,
i.e. respects the inequalities defining the symplectic quotients as subsets
of their Zariski closures, see \cite[Definition 6(ii)]{FarHerSea}.
Note that these inequalities are the restrictions of those describing the orbit spaces,
see \cite{ProcesiSchwarz}.
A \emph{regular} (Poisson) map between symplectic quotients is a smooth
(Poisson) map that arises as such a lift.
Apart from symplectomorphisms,
we will consider various notions of equivalence between symplectic
quotients that are based on this Lifting Theorem, namely
graded regular symplectomorphism and regular diffeomorphism.
For a more detailed exposition of these notions we refer the reader to
Section \ref{sec:Background}.

A natural idea that can be used when studying  a unitary representation $K\to \U(V)$ of a
compact Lie group $K$ on a hermitian vector space $V$ is to complexify. That is, to observe
that the representation extends to
$G:=K_\C\to \GL(V)$ (see e.g \cite{GWSkempfNess}).
Note that the notation $\GL(V)$ will always denote $\GL_\C(V)$.
Our results are based on assuming properties for $V$ as a $G$-module.  We borrow the notions of $V$ being \emph{$k$-principal\/} or \emph{$k$-large\/} from \cite{GWSlifting}. See Definition \ref{def:2Principal}.
Roughly speaking, one can
say that
these conditions hold typically.
For instance, in the case of $G$ connected and
simple and $G$-modules $V$ with $V^G=\{0\}$,   all but finitely many isomorphism classes are
$2$-large (cf. \cite[Corollary 11.6]{GWSlifting}).
 Another condition that almost always holds is that $V$ has an open set of closed orbits.  In this case, we say that $V$ is \emph{stable}.

With these notions we are ready to formulate our first result.

\begin{theorem}
\label{thrm:2Principal}
Let $K$ be a
compact Lie group and $V$ a unitary $K$-module. Assume
that the image of $K$ in
$\GL(V)$ is
connected and
positive dimensional and that
the action of $G$ on $V$ is $2$-principal and stable.
Then there does not exist a symplectomorphism between the symplectic quotient $M_0$
and a linear symplectic orbifold.
\end{theorem}

We observe that the theorem does not require
the symplectomorphisms  to be regular.
Now suppose that $X$ is a
Hamiltonian $K$-manifold and that $x\in M$, the zero set of the moment mapping. Then we have the symplectic slice  $(W,K_x)$ (see \cite[\S 6]{HerbigSchwarz} and references therein) and the symplectic quotient  $M_0$ of $X$ near the image $x_0$ of $x$ is isomorphic to the symplectic quotient $N_0$ of $W$ by $K_x$.
\begin{corollary}
\label{cor:KManifoldConnected}
Let $X$, $M_0$ etc.\ be as above. Suppose that the image of $K_x$ in
$\GL(W)$ is connected and positive dimensional and that the   action of $(K_x)_\C$ on $W$ is stable and $2$-principal. Then $M_0$ is not symplectomorphic to a linear symplectic orbifold in a neighborhood of $x_0$.
\end{corollary}

In the case that the image of $K$ in
$\GL(V)$ is not connected,
we have the following, which under slightly stronger
hypotheses precludes the existence of a regular symplectomorphism to a linear orbifold.

\begin{theorem}
\label{thrm:KNotConnected}
Let $K$ be a compact Lie group and $V$ a unitary $K$-module. Assume
that the image of $K$ in
$\GL(V)$ is positive dimensional and that
the action of $G$ on $V$ is $2$-large.
Then there does not exist a regular symplectomorphism between the symplectic
quotient $M_0$ and a linear symplectic orbifold.
\end{theorem}

In fact, when the action of $G$ on $V$ is $2$-large (or even $1$-large),
the real analytic ideal of the zero set of the moment map is generated by
the component functions of the moment map, see \cite[Corollary 4.3]{HerbigSchwarz}.
The proof of Theorem \ref{thrm:KNotConnected} can be used to show that,
with the same hypotheses, there can be no germ of a real analytic
symplectomorphism from a neighborhood of the origin of $M_0$ (the image
of $0$ in $V$) to a linear symplectic orbifold.
Applying the symplectic slice theorem and the fact that any symplectic
$K$-manifold admits a unique structure as a real analytic symplectic
$K$-manifold \cite{KutzschebauchLoose}, we have the following.
\begin{corollary}
Let $X$, $M_0$ etc.\ be as in Corollary \ref{cor:KManifoldConnected}.
Suppose that the image of $K_x$ in
$\GL(W)$ is positive dimensional and that the
action of $(K_x)_\C$ on $W$ is
$2$-large.
Then $M_0$ is not
real analytically symplectomorphic to a linear symplectic orbifold
in a neighborhood of $x_0$.
\end{corollary}

The next theorem addresses circle representations and applies to representations whose
complexifications are not $2$-principle.

\begin{theorem}
\label{thrm:Circle}
Let $K = \Sp^1$.  Let $V$ be a unitary $K$-module with $V^K = \{ 0 \}$ such that the
corresponding symplectic quotient $M_0$ has real dimension at least $4$.  Then there
does not exist a regular diffeomorphism between $M_0$ and a linear symplectic orbifold.
\end{theorem}

In \cite{HerbigSeaton2}, a weaker version of Theorem \ref{sec:Circle} was presented,
demonstrating that $M_0$ cannot admit a $\Z$-graded regular symplectomorphism to a
linear symplectic orbifold.

In the case that $K = \Sp^1$ and $\dim_\R(M_0) = 2$, an explicit $\Z$-graded regular symplectomorphism
was given in \cite[Section 4.3]{FarHerSea} between $M_0$ and the orbifold $\C/\Z_m$ where $m$
can be computed from the weights.  Combining these results with Theorem \ref{thrm:Circle}
yields a complete answer to the question of which $\Sp^1$-linear symplectic quotients admit
a regular diffeomorphism to a linear symplectic orbifold.  Moreover, it follows that an
$\Sp^1$-linear symplectic quotient admits a regular diffeomorphism to a linear symplectic
orbifold if and only if it admits a $\Z$-graded regular symplectomorphism to a linear
symplectic orbifold.

Finally, in the case of $K=\SU_2$, we are also able to give a complete result.  In this case,
the task of working through the list of non-$2$-principle cases is not too demanding. Recall
that the irreducible modules of $K=\SU_2$ are given by the
spaces $R_d$ of   binary
forms of degree $d$.  We demonstrate the following.

\begin{theorem}
\label{thrm:SU2}
Let $K = \SU_2$, and let $V$ be a
nontrivial
unitary $K$-module
with $V^K=\{0\}$.  Then the corresponding
symplectic quotient $M_0$ is graded regularly symplectomorphic to a linear
symplectic orbifold if and only if $V$ is isomorphic to
$R_1$, $R_1\oplus R_1$, $R_2$,
$R_3$, or $R_4$.
\end{theorem}

The outline of this paper is as follows.  In Section \ref{sec:Background},
we fix notation and recall the definitions and results we will need.  In
Section \ref{sec:2Principal}, we consider $K$-modules $V$ such that the action
of $G$ on $V$ is $2$-principal
or $2$-large
and prove
Theorems \ref{thrm:2Principal} and \ref{thrm:KNotConnected}.
In Section \ref{sec:Circle}, we consider the case $K = \Sp^1$ and prove Theorem
\ref{thrm:Circle}.  We restrict our attention to the case $K = \SU_2$ in
Section \ref{sec:SU2} and inspect each of the representations that are not
addressed by Theorem \ref{thrm:2Principal}, see Equation \eqref{eq:SU2candidates},
proving Theorem \ref{thrm:SU2}.

% xxxxxxxxxxxxxxxxxxxxxxxxxxxxxxxxxxxxxxxxxxxxxxxxxxxxxxxxxxxxxxxxxxxxxxxxx
% xxxxxxxxxxxxxxxxxxxxxxxxxxxxxxxxxxxxxxxxxxxxxxxxxxxxxxxxxxxxxxxxxxxxxxxxx
% xxxxxxxxxxxxxxxxxxxxxxxxxxxxxxxxxxxxxxxxxxxxxxxxxxxxxxxxxxxxxxxxxxxxxxxxx

\section*{Acknowledgements}

We would like to thank Reyer Sjamaar for facilitating this collaboration and
Leonid Bedratyuk for assistance in computing the invariants of $\SL_2$.

% xxxxxxxxxxxxxxxxxxxxxxxxxxxxxxxxxxxxxxxxxxxxxxxxxxxxxxxxxxxxxxxxxxxxxxxxx
% xxxxxxxxxxxxxxxxxxxxxxxxxxxxxxxxxxxxxxxxxxxxxxxxxxxxxxxxxxxxxxxxxxxxxxxxx
% xxxxxxxxxxxxxxxxxxxxxxxxxxxxxxxxxxxxxxxxxxxxxxxxxxxxxxxxxxxxxxxxxxxxxxxxx

\section{Background}
\label{sec:Background}

We use the following notation throughout this paper.
Let $K$ be a compact Lie group and let $V$ be a unitary $K$-module.
Let $J\co V \to \mathfrak{k}^\ast$ denote the homogeneous quadratic moment map
where $\mathfrak{k}$ denotes the Lie algebra of $K$ and $\mathfrak{k}^\ast$ its dual,
and let $M := J^{-1}(0)$ denote the zero-fiber of $J$.  The symplectic quotient of
$V$ by $K$ is given by $M_0 := M/K$.  This space is a \emph{symplectic stratified space},
see \cite{SjamaarLerman}, where the stratification is
by orbit type.
It is a \emph{differential space} with smooth structure given by
$\mathcal{C}^\infty(M_0) := \mathcal{C}^\infty(V)^K/\mathcal{I}_M^K$ where
$\mathcal{I}_M$ denotes the ideal of smooth functions vanishing on $M$
and $\mathcal{I}_M^K := \mathcal{I}_M \cap \mathcal{C}^\infty(V)^K$.
The algebra $\mathcal{C}^\infty(M_0)$ inherits a Poisson bracket $\{\, ,\,\}$
from $\mathcal{C}^\infty(V)$ with respect to which
$(M_0, \mathcal{C}^\infty(M_0), \{\, ,\,\})$ is a \emph{Poisson differential space},
see \cite[Definition 5]{FarHerSea}.
We define the \emph{(Poisson) algebra of regular functions on $M_0$}, denoted $\R[M_0]$,
to be the $\Z$-graded (Poisson) subalgebra of $\mathcal{C}^\infty(M_0)$ given by
the image of $\R[V]^K \subset \mathcal{C}^\infty(V)^K$ via the quotient map
$\mathcal{C}^\infty(V)^K \to \mathcal{C}^\infty(V)^K/\mathcal{I}_M^K$.
Let $\mathcal{J}$ denote the vanishing ideal of $J^{-1}(0)$ in $\R[V]$, and then
$\R[M_0]$ is given by $\R[V]^K/\mathcal{J}^K$
where $\mathcal{J}^K := \mathcal{J} \cap\R[V]^K$.
Note that this coincides
with defining the regular functions on $M_0$ in terms of a \emph{global chart}
defined using a minimal generating set for $\R[V]^K$; see \cite[Definition 7]{FarHerSea}.

A \emph{diffeomorphism} between differential spaces
$(M_0, \mathcal{C}^\infty(M_0))$ and $(N_0, \mathcal{C}^\infty(N_0))$ is a homeomorphism
$\chi\co M_0 \to N_0$
such that pull-back via $\chi$ induces an isomorphism $\chi^\ast\co \mathcal{C}^\infty(N_0) \to \mathcal{C}^\infty(M_0)$.
 A \emph{symplectomorphism} between Poisson differential spaces is a
diffeomorphism $\chi$ such that $\chi^\ast$ is an isomorphism of Poisson algebras.
If $M_0$ and $N_0$
are symplectic quotients of unitary representations equipped with algebras of regular functions
as above,
then we say that a diffeomorphism (or symplectomorphism) is \emph{regular} if $\chi^\ast$
restricts to an isomorphism $\R[N_0] \to \R[M_0]$ and \emph{$\Z$-graded} if
this isomorphism preserves the $\Z$-gradings of $\R[N_0]$ and $\R[M_0]$.
It is easy to see that a ($\Z$-graded) regular symplectomorphism induces an isomorphism
between global charts defined in terms of minimal generating sets as above,
and hence that these definitions coincide with those given in \cite[Section 4]{FarHerSea}.
The analogs of the above concepts can be defined for the real analytic structure on $M_0$
defined in terms of the real analytic functions on $V$.

A \emph{linear symplectic orbifold} is a quotient $W/H$ where $H$ is a finite group
and $W$ is a unitary $H$-module.  Note that in this case, the moment map is trivial
so that $W/H$ is the corresponding symplectic quotient.
By a \emph{linear orbifold}, we mean a quotient
$W/H$ where $H$ is finite and $W$ is a (real) $H$-module.

Let $G := K_\C$ denote the complexification of $K$, and then $V$ is as well a $G$-module.
Let $Z:= V\git G = \operatorname{Spec}(\C[V]^G)$ denote the GIT quotient. Then there is a  Zariski open dense subset $Z_{\pr}$
such that the corresponding closed orbits in $V$ have isotropy group conjugate to a fixed reductive subgroup $H$ of $G$.
We call $H$ a \emph{principal isotropy group\/} of $V$ and we denote the inverse image of $Z_{\pr}$ in $V$ by $V_{\pr}$.
We say that $V$ has  \emph{FPIG} (resp.\  \emph{TPIG\/}) if the principal isotropy group $H$ is  finite (resp.\ trivial).
 We say that $V$ is \emph{stable\/}
 if there is an open dense subset of closed orbits. Equivalently, $V_\pr$ consists of closed orbits.
 Note that FPIG implies stable.

We recall the following; see
\cite[Section 0.3]{GWSlifting}.
For a
$G$-variety $X$, we let $X_{(\ell)}$ denote the set
of points $x \in X$ such that the isotropy group $G_x$ has complex dimension
$\ell$ and define the \emph{modularity} $\modu(X,G)$ to be the maximum
of $\dim_\C X_{(\ell)} - \dim_\C G + \ell$.

\begin{definition}
\label{def:2Principal}
Let $G$ be a complex
reductive
group and $V$ a $G$-module.
We say that $V$ is \emph{$k$-principal} if the complex codimension of
$V\smallsetminus V_{\pr}$ is at least
$k$. We say that $V$ is \emph{$k$-modular}
if
it has FPIG and
$\modu(V\smallsetminus V_{(0)}, G) \leq \dim_\C V\git G - k$.
This is equivalent to the condition that $V$ has FPIG and $\codim V_{(\ell)}\geq k+\ell$ for all $\ell\geq 1$,
see \cite[Remark 9.6(3)]{GWSlifting}.
If $V$ is both $k$-principal and $k$-modular, we say it is \emph{$k$-large}.
\end{definition}

Note that the definition of
$k$-principal differs slightly from that of \cite{GWSlifting}
in that we do not require that $V$ has FPIG.

\begin{remark}\label{rem:StableTPIG}
Let $V$ be a $2$-principal $G$-module. If $V$ is stable, then the principal isotropy group $H$ is the kernel of
$G\to\GL(V)$ \cite[Corollary 7.7(2)]{GWSlifting}.
Hence if $G$ acts faithfully on $V$, then $V$ has TPIG.
\end{remark}

Let $(z_1, \ldots, z_n)$ denote a choice of complex coordinates for $V$.  Using real
coordinates $(z_1, \ldots, z_n, \overline{z_1},\ldots,\overline{z_n})$, the action of
$k \in K$ is given by $\diag(k, (k^{-1})^{t})$.  Complexifying the underlying real structure
yields $V_\C:= V\otimes_\R \C $
which is isomorphic as a
$K$-module to
$V\oplus V^\ast$.
We let $G$ act on $V\oplus V^\ast$ as the complexification of the action of $K$.
Using the corresponding complex coordinates
$(z_1,\ldots,z_n,w_1,\ldots,w_n)$ for $V_\C$, the real points of $V_\C$ are those such that
$\overline{z_i} = w_i$ for each $i$.  Note that the natural isomorphism
$\R[V] \otimes_\R \C \simeq \C[V_\C]$ restricts to an isomorphism from
$\R[V]^K \otimes_\R \C$ to $\C[V_\C]^G$; see \cite[Proposition 5.8]{GWSliftingHomotopies}.

Let
\begin{equation}
\label{eq:CMomentGeneral}
    J_\C\co V\oplus V^\ast\longrightarrow \mathfrak{g}^\ast
\end{equation}
denote the complexification of the moment map, where
$\mathfrak{g} = \mathfrak{k} \oplus i\mathfrak{k}$ is the
Lie algebra of $G$. Let $N := J_\C^{-1}(0) \subset V\oplus V^\ast$,
and let $\mathcal{J}_\C$ denote the vanishing ideal of $N$ in $\C[V\oplus V^\ast]$.
Considering $V$ as a subset of $V_\C$ as above, we have by
\cite[Corollaries 4.2 and 4.3]{HerbigSchwarz} that if $V$ is $1$-large as a $G$-module,
then $M = J^{-1}(0)$ is Zariski dense in $N$
and $\mathcal{J}_\C$ is generated by the component functions of $J$.  Then as $K$ is Zariski
dense in $G$, it follows that $\R[M_0] \otimes_\R\C$ is isomorphic to
$\C[V\oplus V^\ast]^G/\mathcal{J}_\C^G$ where
$\mathcal{J}_\C^G := \mathcal{J}_\C \cap \C[V\oplus V^\ast]^G$.
Hence, $\R[M_0] \otimes_\R\C$ is the coordinate ring of the complex algebraic variety
$N\git G$, which we refer to as the \emph{complex symplectic quotient}.

\begin{lemma}
\label{lem:SympQuotSingular}
If $G$ acts nontrivially and $V\git G$ is not a point, then
$\C[V\oplus V^\ast]^G/\mathcal{J}_\C^G$ is not a polynomial algebra.
Hence if $V$ is $1$-large as a $G$-module, then $\R[M_0]$
is not a polynomial algebra.
\end{lemma}
\begin{proof}
As $J$ is real, it is easy to see that each component of $J_\C$
is contained in the kernels of the projections
$\C[V\oplus V^\ast]\to \C[V]$ and $\C[V\oplus V^\ast]\to \C[V^\ast]$.
Therefore, $\C[V\oplus V^\ast]^G/\mathcal{J}_\C^G$ contains both
$\C[V]^G$ and $\C[V^\ast]^G$, which generate a subalgebra
$\mathcal{A}$ of dimension $2\dim V\git G$.  However, as
the image of the invariant dual pairing in $\C[V\oplus V^\ast]^G/\mathcal{J}_\C^G$
is easily seen to not be an element of $\mathcal{A}$,
it follows that the number of minimal generators of
$\C[V\oplus V^\ast]^G/\mathcal{J}_\C^G$ is greater than its dimension.
\end{proof}

To conclude this section, we demonstrate the following elementary fact that
we will use several times in the sequel.

\begin{lemma}
\label{lem:Codim2Homotopy}
Let $U$ be a piecewise linear $n$-dimensional submanifold of a real vector space
with $n\geq 4$.
Let $C$ be a real algebraic subset of $U$ with $\dim_\R C \leq n - 4$.
Let $\mathcal{O}:= U\smallsetminus C$ and choose a base point in $\mathcal{O}$. Then the inclusion $\mathcal{O}\to U$ induces isomorphisms
$\pi_1(\mathcal{O}) \simeq \pi_1(U)$
and
$\pi_2(\mathcal{O}) \simeq \pi_2(U)$.
\end{lemma}
\begin{proof}
Let $k \in \{1,2\}$, let $f\co \Sp^k \to \mathcal{O}$ be a continuous map from the
$k$-sphere into $\mathcal{O}$
that extends to a map $\overline{f}\co\mathbb{D}^{k+1}\to U$
on the unit disk.  Up to homotopy, we may assume that
$\overline{f}$ is piecewise linear so that
$P:=\overline{f}(\mathbb{D}^{k+1})$ is a polyhedron with subpolyhedron
$P_0:=f(\Sp^k)$.  Note that we may assume that $C$ is a polyhedron as well by
\cite[Theorem 2.12]{ParkSuhTriang}.  Then by \cite[Section 5.3]{RourkeSanderson},
there is an isotopy of $U$ that fixes $P_0$ and finishes with a map
$h\co U\to U$ such that $h(P - P_0) \cap C = \emptyset$.  Therefore,
$P \subset \mathcal{O}$, from which we conclude that $f$ is trivial in $\pi_k(\mathcal{O})$.
Hence $\pi_k(\mathcal{O})\to\pi_k(U)$ is injective. Similarly, given $f\co\Sp^k\to U$ we may find a homotopy
to an $f'\co\Sp^k\to\mathcal{O}$ preserving the base point. Hence $\pi_k(\mathcal{O})\simeq\pi_k(U)$.
\end{proof}

% xxxxxxxxxxxxxxxxxxxxxxxxxxxxxxxxxxxxxxxxxxxxxxxxxxxxxxxxxxxxxxxxxxxxxxxxx
% xxxxxxxxxxxxxxxxxxxxxxxxxxxxxxxxxxxxxxxxxxxxxxxxxxxxxxxxxxxxxxxxxxxxxxxxx
% xxxxxxxxxxxxxxxxxxxxxxxxxxxxxxxxxxxxxxxxxxxxxxxxxxxxxxxxxxxxxxxxxxxxxxxxx

\section{Symplectic quotients associated to $2$-principal actions}
\label{sec:2Principal}

In this section, we demonstrate Theorem \ref{thrm:2Principal}, that if $K$
has connected, positive dimensional image in
$\GL(V)$ and $V$ is a $2$-principal stable
$(G = K_\C)$-module, then the corresponding
symplectic quotient $M_0$ does not admit a symplectomorphism to a linear symplectic orbifold.
We also demonstrate Theorem \ref{thrm:KNotConnected}, which applies to the case that
the image of $K$ in
$\GL(V)$ is not connected and indicates hypotheses under which
$M_0$ does not admit a regular symplectomorphism to a linear symplectic orbifold.
First, we make some observations about the associated Hilbert embedding
in the general context introduced in Section \ref{sec:Background}.

Let $V$ be a $G$-module.  Choose a minimal
homogeneous generating set $p_1, \ldots, p_d$ for $\C[V]^G$, let
$p = (p_1, \ldots, p_d)\co V \to \C^d$ denote the corresponding Hilbert map,
and identify the image $p(V)\subset \C^d$ of $p$ with the
GIT quotient $Z:=V\git G$.  Then the restriction of $p$ to $M$ induces the \emph{Kempf-Ness
homeomorphism} $\overline{p} \co M/K \to Z$; see \cite{KempfNess,GWSkempfNess}.
We recall the following.

\begin{lemma}
\label{lem:KempfNessStratPres}
The mapping $\overline{p}\co M/K \to V\git G$ is an orbit type stratum-preserving homeomorphism where
the orbit type in $M/K$ corresponding to the isotropy group $L \leq K$ is mapped to the
orbit type in $V\git G$ corresponding to the isotropy group $L_\C \leq G$, the complexification of $L$.
\end{lemma}
\begin{proof}
By \cite[Proposition 1.3(ii)]{DadokKac}, we have that if $v \in M$, then $G_v = (K_v)_\C$.
Then by \cite[Proposition 5.8(3)]{GWSliftingHomotopies}, the $G$-conjugacy
classes of $G$-isotropy groups coincide with the $K$-conjugacy classes of $K$-isotropy
groups.
\end{proof}

Now assume that the action of $G$ on $V$ is $2$-principal and
stable.
Then $Z\smallsetminus Z_{\pr}$ has complex codimension
$2$ in $Z$, see \cite[(Section 6.2 and Remark 9.6(4)]{GWSlifting}.  Letting $(M/K)_{\pr}$
denote the set of orbits in $M/K$ with principal isotropy, Lemma \ref{lem:KempfNessStratPres}
implies that $\overline{p}$ restricts to a homeomorphism $(M/K)_{\pr}\to Z_{\pr}$,
from which we have the following.

\begin{corollary}
\label{cor:2PrincCodim}
Suppose that the action of $G$ on $V$ is $2$-principal and
stable.
Then $(M/K)\smallsetminus(M/K)_{\pr}$ has real codimension at least $4$.
\end{corollary}

Similarly, we have the following.

\begin{lemma}
\label{lem:2PrincSimpConn}
Suppose that the action of $G$ on $V$ is $2$-principal and stable.
Then
$\pi_1((M/K)_{\pr})\simeq\pi_0(G)\simeq\pi_0(K)$.
\end{lemma}
\begin{proof}
By Remark \ref{rem:StableTPIG}, replacing $G$ by its image in
$\GL(V)$, we reduce to
the case that $V$ has TPIG.
Then $V_{\pr} \to Z_{\pr}$ is a principal $G$-bundle,
see \cite[Theorem 6.10(3)]{PopovVinberg}.
From the exact sequence of a fibration, we obtain the exact sequence
\[
    \pi_1(V_{\pr}) \to \pi_1(Z_{\pr}) \to \pi_0(G)
    \to \pi_0(V_{\pr}).
\]
By Lemma  \ref{lem:Codim2Homotopy}, $\pi_1(V_{\pr})\simeq\pi_1(V)$ is trivial, and
$\pi_0(V_{\pr})$ is clearly trivial as well so that $\pi_1(Z_{\pr}) \to \pi_0(G)$
is a bijection.
As $(M/K)_{\pr}$ is homeomorphic to $Z_{\pr}$ by Lemma \ref{lem:KempfNessStratPres},
the claim follows.
\end{proof}

\begin{proof}[Proof of Theorem \ref{thrm:2Principal}]
We again apply Remark \ref{rem:StableTPIG} to reduce to the case that $V$ has TPIG.
Assume
that there is a symplectomorphism
$\chi\co M_0 \to W/H$ where $H$ is a finite group and $W$ a unitary
$H$-module.  Let $Y:= W/H$, and then as
$\chi$ is a symplectomorphism and hence preserves
connected components of
orbit type strata by
\cite[Proposition 3.3]{SjamaarLerman}, we have $\chi((M/K)_{\pr}) = Y_{\pr}$
is the
set of principal orbits  of $Y$.
It follows from Corollary
\ref{cor:2PrincCodim} that $Y\smallsetminus Y_{\pr}$ has
real codimension
at least $4$ in $Y$ and hence that $W\smallsetminus W_{\pr}$ has
real
codimension at least $4$ in $W$, where $W_{\pr}$ denotes the
set of principal orbits.
By Lemma \ref{lem:Codim2Homotopy}, $\pi_1(W_{\pr})$ is trivial,
and hence
$\pi_1(Y_{\pr}) \simeq H$.

Recalling that $\chi$ restricts to a homeomorphism from
$(M/K)_{\pr}$ to
$Y_{\pr}$, it follows that $\pi_1((M/K)_{\pr}) = \pi_1(Y_{\pr}) = H$.  As
$(M/K)_{\pr}$ is simply connected by Lemma \ref{lem:2PrincSimpConn}
and the fact that the image of $K$ in
$\GL(V)$ is connected,
it follows that
$H$ is trivial.  But then $Y_{\pr} = Y$ so that $M/K = (M/K)_{\pr}$, implying that
$K$
acts trivially on $V$.
This yields a contradiction, completing the proof.
\end{proof}

\begin{remark}
We observe that the proof of Theorem \ref{thrm:2Principal} demonstrates the stronger fact
that with the same hypotheses, even an orbit type stratum-preserving homeomorphism to a
linear orbifold, which need not be symplectic, does not exist.
\end{remark}

\begin{proof}[Proof of Theorem \ref{thrm:KNotConnected}]
As above, we reduce to the case that $V$ has TPIG by Remark \ref{rem:StableTPIG}.
Assume
there is a regular symplectomorphism $\chi\co M_0 \to W/H$ with $H$ finite and $W$ a
unitary $H$-module, and then we have a Poisson isomorphism $\chi^\ast\co\R[W]^H\to \R[M_0]$.
As in the proof of Theorem \ref{thrm:2Principal}, $\chi$ preserves connected components of
orbit type strata and hence restricts to a regular diffeomorphism
$(M/K)_{\pr}\to (W/H)_{\pr}$.  Similarly, as $(M/K)\smallsetminus(M/K)_{\pr}$
has real codimension at least four in $M/K$ so that $W\smallsetminus W_{\pr}$
has real codimension at least four in $W$, we have again that $W_{\pr}$ is simply
connected and $\pi_1((W/H)_{\pr}) \simeq H$.
By Lemma \ref{lem:2PrincSimpConn}, $\pi_1((M/K)_{\pr})\simeq\pi_0(K)$.
Since $(M/K)_{\pr}$ and $(W/H)_{\pr}$
are homeomorphic, we have that $\Gamma:=\pi_0(K)\simeq H$.

Recall from Section \ref{sec:Background} that
by \cite[Corollary 4.2]{HerbigSchwarz},
$M$ is Zariski dense in $N$
so that tensoring with $\C$ yields an isomorphism
$\chi_\C^\ast\co\C[W\oplus W^\ast]^H\to \C[V\oplus V^\ast]^G/\mathcal{J}_\C^G$.
Letting $X:=N\git G$ denote the complex symplectic quotient and
$Y:=(W\oplus W^\ast)/H$, we then have an isomorphism $\chi_\C\co X \to Y$.
The complexifications of the symplectic forms on $W$ and $V$ yield (complex)
symplectic forms on $W\oplus W^\ast$ and $V\oplus V^\ast$, respectively,
and one checks that the corresponding Poisson brackets on
$\C[W\oplus W^\ast]^H$ and $\C[V\oplus V^\ast]^G$ are the complexifications
of the respective Poisson brackets on $\R[W]^H$ and $\R[V]^K$.  Moreover,
the ideals of $M$ and $N$ are generated by
the component functions of
the appropriate
$J$ by
\cite[Corollary 4.3]{HerbigSchwarz} so the Poisson bracket on $\C[V\oplus V^\ast]^G$
induces a well-defined Poisson bracket on $\C[V\oplus V^\ast]^G/\mathcal{J}_\C^G$
with respect to which $\chi_\C^\ast$ is by construction a Poisson isomorphism.

We claim that the respective Poisson structures on
$\C[V\oplus V^\ast]^G/\mathcal{J}_\C^G$ and $\C[W\oplus W^\ast]^H$
determine $X_{\pr}$ and $Y_{\pr}$, which we demonstrate using
techniques similar to those of \cite{Gonc}.
Let $f \in \C[V\oplus V^\ast]^G$ and let $(v,w) \in N$
with closed $G$-orbit and isotropy group $L$.  Then $df(v,w)$ is fixed by $L$
so that the Hamiltonian vector field of $f$ is tangent to $(V\oplus V^\ast)^L$
at $(v,w)$.  If $(v,w)$ has (trivial) principal isotropy group, then the
differentials of the components of the moment map are linearly independent
at $(v,w)$, which is therefore a smooth point of $N$.  Then the slice representation
at $(v,w)$ is trivial, see \cite{LunaSlice}, so that the Hamiltonian vector fields
span the tangent space of $X$ at the orbit of $(v,w)$.  It follows that the Hamiltonian
flows do not leave $X_{\pr}$,
and every pair of points in $X_{\pr}$ can be connected by broken flow lines so that
$X_{\pr}$ is determined by the Poisson structure on $\C[V\oplus V^\ast]^G/\mathcal{J}_\C^G$.
Similarly, the Poisson structure on $\C[W\oplus W^\ast]^H$ determines $Y_{\pr}$ so that
$\chi_\C$
restricts to an isomorphism $X_{\pr}\to Y_{\pr}$.

Now,
by inspecting the action of $H$ on $W\oplus W^\ast$, it is easy to see that
$(W\oplus W^\ast)\smallsetminus (W\oplus W^\ast)_{\pr}$ has real
codimension at least $4$ in $W\oplus W^\ast$.  Therefore,
$(W\oplus W^\ast)_{\pr}$ is simply connected by Lemma \ref{lem:SympQuotSingular}.
Then $(W\oplus W^\ast)_{\pr} \to Y_{\pr}$ is a principal $H$-bundle and hence
the universal covering of $Y_{\pr}$ so that $\pi_1(Y_{\pr})\simeq H$.
As $Y_{\pr}$ and $X_{\pr}$ are homeomorphic, $\pi_1(X_{\pr})\simeq H$ as well.

Let
$G^0$ denote the connected component of the identity in $G$, let
$X_0 := N\git G^0$, let $N_{\pr}$ denote the
union of the principal orbits of $N$ as
a $G$-variety, and let $(X_0)_{\pr}:=N_{\pr}\git G^0$.  Then
$(X_0)_{\pr} \to X_{\pr}$ is a principal $\Gamma$-bundle and hence a covering space.
Therefore,
there is a covering map $\phi$ that makes the following
diagram commute:
\[
    \xymatrix{
        (X_0)_{\pr}
            \ar[d]_{/\Gamma}
            \ar@{<.}[r]^{\phi}
        &(W\oplus W^\ast)_{\pr}
            \ar[d]^{/H}
        \\
        X_{\pr}
            \ar[r]^{\chi_\C}
        &
        Y_{\pr}
    }
\]
However, recalling that $\Gamma \simeq H$, it follows that $\phi$ is an isomorphism.

By \cite[Theorem 2.2]{HerbigSchwarz}, $N$ is normal so that $X_0$ is normal by
\cite[Proposition 5.4.1]{BrunsHerzog}.
The codimension of $(X_0)\smallsetminus (X_0)_{\pr}$ is
at least two, and thus
$\phi$
extends to an isomorphism $W\oplus W^\ast \to X_0$.  Then clearly,
$X_0$ is smooth.  However, noting that the moment map of the $K$-action on $V$
coincides with that of the $K^0$-action, $X_0$ is the complex symplectic
quotient associated to the $K^0$-module $V$.  This contradicts Lemma
\ref{lem:SympQuotSingular} and completes the proof.
\end{proof}

% xxxxxxxxxxxxxxxxxxxxxxxxxxxxxxxxxxxxxxxxxxxxxxxxxxxxxxxxxxxxxxxxxxxxxxxxx
% xxxxxxxxxxxxxxxxxxxxxxxxxxxxxxxxxxxxxxxxxxxxxxxxxxxxxxxxxxxxxxxxxxxxxxxxx
% xxxxxxxxxxxxxxxxxxxxxxxxxxxxxxxxxxxxxxxxxxxxxxxxxxxxxxxxxxxxxxxxxxxxxxxxx

\section{Symplectic quotients by $\Sp^1$}
\label{sec:Circle}

Throughout this section, we restrict to $K = \Sp^1$ so that $G = \C^\times$ and
we
assume that $n:= \dim_\C V \geq 3$.
Choose a basis for $V$ with respect to which the $\Sp^1$-action
is diagonal.
Then we may identify $V$ with $\C^n$ and describe the
action of $G$ with the \emph{weight vector} $A = (a_1, \ldots, a_n)$.  We assume with
no loss of generality that each weight $a_i$ is nonzero, for otherwise we may decompose
$V$ into a product with a trivial
module and restrict to the nontrivial factor.
We may assume $V$ is stable as a $G$-module, for if not, then all weights have the same
sign and $V\git G$ (and hence $M_0$) is a point.
Similarly, we assume that $\gcd(a_1, \ldots, a_n) = 1$, for otherwise we may replace
$K$ with $K/(\Z/\gcd(a_1,\ldots,a_n)\Z)$; see \cite{HerbigSeaton2}.
As $V$ is stable, this latter assumption amounts to assuming that $V$ has
TPIG as a $G$-module.
It follows by \cite[Theorem 3.2]{HerbigSchwarz} that the $G$-module $V$ is $1$-large.

Identifying the dual $\mathfrak{k}^\ast$ of the Lie algebra $\mathfrak{k}$ with $\R$,
the moment map $J$ is given by $J(z_1, \ldots, z_n) = \frac{1}{2}\sum_{i=1}^n a_i z_i \overline{z_i}$,
see \cite{HerbigIyengarPflaum}.
Note that if $A$ contains at least two positive and
two negative weights, then $M_0$ is not a rational homology manifold and cannot
be homeomorphic to a symplectic orbifold; see
\cite[Proposition 3.1]{HerbigIyengarPflaum} or \cite[Theorems 3 and 4]{FarHerSea}.
Otherwise, as $A$ has $n-1$ weights of the same sign, the null cone has complex
codimension $1$ in $V$, which is therefore not $2$-principal as a $G$-module, and Theorem
\ref{thrm:2Principal} does not apply.

Recall that $\mathcal{J}$ denotes the vanishing ideal of $M = J^{-1}(0)$ in $\R[V]$.
As not all weights have the same sign,
$\mathcal{J}$ is generated by $J$, see \cite{ArmsGotayJennings,HerbigIyengarPflaum};
this follows as well from \cite[Corollary 4.3]{HerbigSchwarz} and the fact
that the $G$-module $V$ is $1$-large.
Note that the moment map is $K$-invariant so that $\mathcal{J}^K$ is
as well generated by $J$ in this case.  It follows that the algebra $\R[M_0]$ of
real regular functions is given by $\R[V]^K/J\cdot \R[V]^K.$

Using complex coordinates $(\bs{z},\bs{w}):=(z_1, \ldots, z_n, w_1, \ldots, w_n)$
for $V_\C \simeq V\oplus V^\ast$ as in Section \ref{sec:Background}, the
weight vector for the $G$-module $V\oplus V^\ast$ is
$(a_1, \ldots, a_n, -a_1, \ldots, -a_n)$.  Similarly, the complexified moment map
$J_\C\co V\oplus V^\ast\longrightarrow \mathfrak{g}^\ast$ is given by
\begin{equation}
\label{eq:S1CMoment}
    (\bs{z},\bs{w}) \longmapsto \sum_{i=1}^n a_i z_i w_i
\end{equation}
Recall that $N := J_\C^{-1}(0) \subset V\oplus V^\ast$
and $\mathcal{J}_\C$ is the
vanishing ideal of $N$ in $\C[V\oplus V^\ast]$.
As $V$ is $1$-large, we have by \cite[Theorem 2.2]{HerbigSchwarz} that $N$
is irreducible.  Moreover, $\R[M_0] \otimes_\R\C$ is isomorphic to the
coordinate ring $\C[V\oplus V^\ast]^G/\mathcal{J}_\C^G$ of the complex
symplectic quotient $X:=N\git G$,
see Section \ref{sec:Background}.
Let $X^\prime := (V\oplus V^\ast) \git G$.

Let $(V\oplus V^\ast)_{\pr}$
and $X_{\pr}^\prime \subset X^\prime$ denote the
sets of principal
orbits.  Let $N_{\pr} := N \cap (V\oplus V^\ast)_{\pr}$, and set
$X_{\pr}:= N_{\pr}\git G$.
For an algebraic variety $Y$, let $Y_{\sm}$ denote the smooth points of $Y$.

\begin{lemma}
\label{lem:S1CIsotropyTypes}
Let $K = \Sp^1$ and assume that the unitary $K$-module $V$ has TPIG as a
$G = K_\C$-module and $n = \dim_\C V \geq 3$.
Then $N\smallsetminus N_{\pr}$ has complex codimension at least $2$ in $N$, and
$X\smallsetminus X_{\pr}$ has complex codimension at least $2$ in $X$.
Moreover, $X_{\sm}=X_{\pr}$.
\end{lemma}
 \begin{proof}
Since the weight vector for the $G$-action on $V\oplus V^\ast$ has $n$ positive and $n$ negative weights,
the null cone has complex codimension $n\geq 3$. Hence it intersects $N$ in codimension
at least
two.  Let $L$ be the isotropy group of a nonprincipal
nonzero closed orbit in $N$. Then its fixed point set is of the form $U\oplus U^*$ where $U$ is the span of $m$ of the $z_i$, $1\leq m\leq n-1$. If $m<n-1$, then
$V^L$ intersects $N$ in codimension at least three.
If $m=n-1$, then
as $J_\C$ restricted to $V^L$ is nontrivial,
$V^L\cap N$ has  codimension
at least two in $N$.
Hence $N\smallsetminus N_{\pr}$ has codimension
at least two in $N$.
Since the inverse image of  $X\smallsetminus X_\pr$ is $N\smallsetminus N_{\pr}$, we must have that $X\smallsetminus X_\pr$ has codimension
at least two in $X$.

Let $x\in N$ be on a nonzero closed orbit with isotropy group $L\neq\{e\}$. Then the
nontrivial part of the
slice
representation of $L$, see \cite{LunaSlice},
is of the form
$L\to\GL(U\oplus U^*)$
where $L$ is finite and acts nontrivially on $U$. Thus the image of $L$ in
$\GL(U\oplus U^*)$ is not generated by pseudoreflections and by the converse to Chevalley's theorem, $X$ is not smooth at the image of $x$.
Hence, except perhaps for the image $x_0$ of $0\in N$, we have that $X_{\pr}=X_{\sm}$.
But $x_0$ is a singular point of $X$ by Lemma \ref{lem:SympQuotSingular}
so that $X_{\pr}=X_{\sm}$.
\end{proof}

With this, we proceed with the following.

\begin{proof}[Proof of Theorem \ref{thrm:Circle}]
Assume
that there is a linear symplectic orbifold $W/H$ and a regular
diffeomorphism $\chi\co M_0 \to W/H$ so that $\chi^\ast\co\R[W]^H\to \R[V]^K/\mathcal{J}^K$
is an isomorphism.  Tensoring with $\C$ yields an isomorphism
$\chi_\C^\ast\co\C[W\oplus W^\ast]^H\to \C[V\oplus V^\ast]^G/\mathcal{J}_\C^G$ and hence
a corresponding isomorphism $\chi_\C\co X \to Y:=(W\oplus W^\ast)/H$. As above, one has that
$Y_{\sm}=Y_{\pr}$,
see \cite[Theorem 9.12]{GWSlifting}.
Hence $\chi_\C$ induces an isomorphism of $X_{\pr}$ and $Y_{\pr}$.
Since the complement of $(W\oplus W^*)_{\pr}$ has complex codimension two in $(W\oplus W^\ast)$,
$(W\oplus W^\ast)_{\pr}$ is simply connected and $\pi_1(Y_\pr)\simeq H$ is finite. Similarly, $\pi_2(Y_{\pr})=0$.

The quotient mapping $N_{\pr}\to X_{\pr}$ is a $G$-fibration
and we have the exact sequence
\[
    \underbrace{\pi_2(G)}_{=0}
    \to     \pi_2(N_{\pr})
    \to     \underbrace{\pi_2(X_{\pr})}_{=0}
    \to     \underbrace{\pi_1(G)}_{=\Z}
    \to     \pi_1(N_{\pr})
    \to     \underbrace{\pi_1(X_{\pr})}_{=H}
\]
from which we conclude that $\pi_2(N_{\pr}) = 0$ and that
$\pi_1(N_{\pr})$ is infinite.

From Equation \eqref{eq:S1CMoment}, one easily checks that the origin is the only singular
point of $N$, and moreover that the set of regular points $N^\ast:= N\smallsetminus \{0\}$
is invariant under the standard action of $G$ on $V\oplus V^\ast$ by scalar
multiplication.  As $N^\ast$ is smooth and $N^\ast \smallsetminus N_{\pr}$
has codimension $2$ in $N^\ast$, we have by Lemma \ref{lem:Codim2Homotopy} that
$\pi_2(N^\ast) = \pi_2(N_{\pr}) = 0$ and
$\pi_1(N^\ast) = \pi_1(N_{\pr})$ is infinite.
We have a fibration $G \to N^\ast \to Q$ where $Q$ is a smooth hypersurface
of $\C\mathbb{P}^{2n-1}$.  Recalling that $n\geq 3$, we have by the Lefschetz hyperplane
section theorem that $\pi_2(Q) = \pi_2(\C\mathbb{P}^{2n-1}) = \Z$ and
$\pi_1(Q) = \pi_1(\C\mathbb{P}^{2n-1}) = 0$.  From the homotopy exact sequence for fibrations,
we have
\[
    \underbrace{\pi_2(N^\ast)}_{=0}
    \to     \underbrace{\pi_2(Q)}_{=\Z}
    \stackrel{k}{\to}
            \underbrace{\pi_1(G)}_{=\Z}
    \to     \pi_1(N^\ast)
    \to     \underbrace{\pi_1(Q)}_{=0}
\]
where $k$ is a nonnegative integer.  But then $k > 0$ so that $\pi_1(N^\ast)$
is finite, yielding a contradiction and completing the proof.
\end{proof}

% xxxxxxxxxxxxxxxxxxxxxxxxxxxxxxxxxxxxxxxxxxxxxxxxxxxxxxxxxxxxxxxxxxxxxxxxx
% xxxxxxxxxxxxxxxxxxxxxxxxxxxxxxxxxxxxxxxxxxxxxxxxxxxxxxxxxxxxxxxxxxxxxxxxx
% xxxxxxxxxxxxxxxxxxxxxxxxxxxxxxxxxxxxxxxxxxxxxxxxxxxxxxxxxxxxxxxxxxxxxxxxx

\section{Symplectic quotients by $\SU_2$}
\label{sec:SU2}

In this section, we consider $K = \SU_2$ and hence $G = \SL_2(\C)$.
Let $R_d$ denote the unitary
$K$-module
of binary forms
of degree $d$.  It is well known that every irreducible
complex $K$-module
  is isomorphic
to $R_d$ for a positive integer $d$; see
\cite[Section 15.6]{WignerBook} or \cite[Section 4.2]{BerndtBook}. The same classification holds for irreducible $G$-modules.

We consider nontrivial $G$-modules $V$ with $V^G=\{0\}$. Then, up to isomorphism, the non $2$-principal cases are given by the following list.
\begin{equation}
\label{eq:SU2candidates}
    kR_1 \;\mbox{for}\; 1 \leq k \leq 2, \quad
    R_2, \quad
    2R_2, \quad
    R_2\oplus R_1,\quad
    R_3,\quad\mbox{and}\quad
    R_4.
\end{equation}
In particular, \cite[Theorem 11.9]{GWSlifting} demonstrates that
every $G$-module not on this list is
$2$-large except for $3R_1$.
Recall that a
module which is $2$-large must be $2$-principal and have FPIG.
Though $3R_1$ is merely $1$-large, one easily checks that it is still $2$-principal.
Hence, by Theorem \ref{thrm:2Principal}, only the symplectic quotients associated to
the modules on the list may admit a symplectomorphism to a linear symplectic orbifold.

For a unitary $G$-module $V$, we let $\mathcal{J}$ denote the vanishing ideal of the moment map
$J \co V\to\mathfrak{k}^\ast$ in $\R[V]$ and $\mathcal{J}^K = \mathcal{J}\cap \R[V]^K$ as above.
We note further that by \cite[Theorems 2.2 and 3.4]{HerbigSchwarz}, every
nontrivial
representation
of $G$ is
$1$-large,
implying that the components of the moment map generate $\mathcal{J}$, except
for $R_1$, $2R_1$, and $R_2$.

To prove Theorem \ref{thrm:SU2}, we will determine which representations
listed in Equation \eqref{eq:SU2candidates} admit $\Z$-graded regular symplectomorphisms
to linear symplectic orbifolds.  We first consider the cases for which such a
symplectomorphism does not exist.

First recall the following.  Let $W$ be a real representation of the finite group $H$,
and suppose $W$ decomposes into irreducibles as $W = \bigoplus_{j=1}^s \mu_j W_j$
where the $W_j$ are
pairwise nonisomorphic. Then the quadratic invariants in
$\R[W]^H$ are naturally in bijection with the symmetric $2$-tensors in the irreducible
representations.  In particular, there are $\sum_{j=1}^s {\mu_j+1\choose 2}$
linearly independent quadratic
invariants.

\begin{proposition}
\label{prop:SU2NonOrbifolds}
Let $K = \SU_2$ and $G = \SL_2$, and let $V$ be a unitary $K$-module that is isomorphic to
$2R_2$ or $R_2\oplus R_1$.  Then the symplectic quotient $M_0$ associated to $V$
does not admit a $\Z$-graded regular symplectomorphism to a linear symplectic orbifold.
\end{proposition}
\begin{proof}
$\bs{2R_2}$:  Let $V = 2R_2$.  It is easy to see that the scalar $-1\in\SU_2$ acts trivially
on $V$.  Noting that $R_2$ is isomorphic as an $\SL_2$-module to the adjoint representation
of $\SL_2$, it then follows that $V$ is isomorphic as a real
$\SU_2/\{\pm 1\} \simeq \operatorname{SO}_3$-module
to $4\R^3$ where each $\R^3$ carries the standard representation of $\operatorname{SO}_3$.
By the first fundamental theorem of invariant theory for $\operatorname{SO}_3$,
see \cite[Section 9.3]{PopovVinberg}, $\R[4\R^3]^{\operatorname{SO}_3}$ contains ten
linearly independent
quadratic invariants given by the invariant scalar products.
As the moment map is equivariant so that the components of the moment map are not
invariant, the quadratic invariants
remain
linearly independent in the ring $\R[M_0]$ of regular functions.

Assume
that there is a finite group $H$, a (complex) $3$-dimensional unitary
$H$-module $W$, and a $\Z$-graded regular
symplectomorphism from $M_0$ to $W/H$. Then $\R[W]^H$
must as well contain ten
linearly independent quadratic invariants.  Considering the possible decompositions of $W$
into irreducible real representations, it must be that $W \simeq 3W_1 \oplus 2W_2 \oplus W_3$
where the $W_j$ are pairwise
nonisomorphic and $1$-dimensional.  But this is not the decomposition
of a unitary $W$, yielding a contradiction.

$\bs{R_2\oplus R_1}$:   Let $V = R_2\oplus R_1$.
In this case, one computes that there are four
linearly  independent
quadratic generators of $\R[V]^K$, which as above remain
linearly independent in $\R[M_0]$.
Assume
that there is a finite group $H$, a (complex) $2$-dimensional unitary
$H$-module $W$, and a $\Z$-graded regular diffeomorphism from $M_0$ to $W/H$.
Then $\R[W]^H$ contains
four linearly independent quadratic generators.
As a real representation,
we either have that $W$ decomposes into four distinct $1$-dimensional representations or
$W \simeq 2W_1\oplus W_2$ with $W_1$ of dimension $1$ and $W_2$ of dimension $2$.
The former case is not the real decompositions of a unitary $W$,
so consider the latter case.
Then $H$ acts on $W_1$ as multiplication by $-1$ and on $W_2$ as a cyclic group of rotations.
Therefore, $W/H$
has two real codimension two strata given by the images of $(2W_1)/H$
and $W_2/H$.  On the other hand, $M_0$ has only one codimension
two stratum
corresponding to the fixed points of $\Sp^1  < \SU_2$.
As a symplectomorphism preserves orbit type strata
by \cite[Proposition 3.3]{SjamaarLerman} as above, we have reached a contradiction.
\end{proof}

\begin{remark}
\label{rem:SU2LinOrb}
Note further that the arguments given in Proposition \ref{prop:SU2NonOrbifolds} along with
observations about the stratifications can be used to demonstrate that none of the symplectic
quotients considered admit orbit type stratum-preserving
$\Z$-graded regular diffeomorphisms to linear
orbifolds (which need not be symplectic or arise from unitary representations).
\end{remark}

We now turn to the representations listed in Equation \eqref{eq:SU2candidates} that
do admit a $\Z$-graded regular symplectomorphism to an orbifold.

\begin{proposition}
\label{prop:SU2R4Orbifold}
The symplectic quotient associated to $R_4$ is $\Z$-graded regularly symplectomorphic to
$\C^2/H$ where $H\simeq S_3$, the symmetric group on $3$-letters, with the standard diagonal action on $\C^2\simeq \R^2\oplus\R^2$.
\end{proposition}
\begin{proof}
  The complex $K$-module $R_4$ is $U\otimes_\R\C$ where $U$ is a real $\SO_3(\R)\simeq K/(\pm I)$-module. One can identify $U$ with the set of real  trace 0 symmetric $3\times3$ matrices with the conjugation action of $\SO_3(\R)$.  It is well-known that the principal isotropy group $L$ of $U$ in $\SO_3(\R)$ consists of the diagonal elements with entries $\pm 1$, hence is isomorphic to $\Z/2\Z\oplus \Z/2\Z$. The fixed point set of $L$ is 2-dimensional and $H:=N_{\SO_3(\R)}(L)/L\simeq S_3$ with its standard action on $\R^2\simeq U^L$. Similarly, considering the $\SO_3(\C)$-action on $R_4$, we have the same principal isotropy group and corresponding action of $H$ on $(R_4)^L\simeq\C^2$.  By the Luna-Richardson theorem \cite{LunaRichardson} we have isomorphisms $\C[(R_4)^L]^{H}\simeq\C[R_4]^G$ and $\R[U^L]^{H}\simeq\R[U]^K$.

As usual, let $M$ denote the zeroes of the moment mapping on $R_4$.  Then we have a homeomorphism of $M/K$ with $R_4\git G$. It then follows from the Luna-Richardson theorem  that the inclusion  $M^L\to (R_4)^L$ induces a homeomorphism of $M^L/H$ with $(R_4)^L/H$. This implies that  $M^L=(R_4)^L$ and  that $\R[M]^K$ restricts injectively to $\R[(R_4)^L]^{H}$. Now by Weyl \cite[Ch.\ 2 \S 3]{Weyl}, the $H=S_3$ invariants on $(R_4)^L\simeq \R^2\oplus\R^2$ are generated by the polarizations of the invariants on a single copy of $\R^2$. Such invariants all lift to $\R[U\oplus U]^K$ since $\R[U]^K\simeq\R[\R^2]^{H}$. Thus $M_0=M/K$ is regularly $\Z$-graded diffeomorphic to $(R_4)^L/H$.

We need to show that $M_0$ and
$(R_4)^L/H$ are symplectomorphic. But this is automatic. The $K$-invariant symplectic form on $R_4$ restricts to an $H$-invariant symplectic form on $(R_4)^L$
by \cite[Lemma 27.1]{GuilSternSTPhysics}.
If $f$ and $f'$ are $K$-invariant smooth functions on $R_4$, then on
$(R_4)^L$ the differentials of the functions have to be
$L$-fixed. Hence the Poisson bracket of $f$ and $f'$ on $R_4$ restricts to the Poisson bracket of the restrictions of $f$ and $f'$ to
$(R_4)^L$.
\end{proof}
\begin{remark}\label{rem:R_2}
Suppose that $V=R_2$. Then a similar (and easier) argument shows that $M_0$ is $\Z$-graded regularly symplectomorphic to $\C/(\pm 1)$; see also \cite{GotayBos}.
\end{remark}

One cannot use the Luna-Richardson theorem to get the desired result in the case of $R_3$ since the normalizer of a principal isotropy group
is not finite.  Hence, we instead use the following.

Recall \cite{DadokKac}, see also \cite{Dadok} or \cite[Section 8.6]{PopovVinberg},
that if $G$ is a connected reductive group, then the complex $G$-module $V$ is \emph{polar}
if there is a \emph{Cartan subspace} of $V$, a
linear subspace $\mathfrak{c}$ of $V$ such that
each element of $\mathfrak{c}$ has closed $G$-orbit, $\dim_\C\mathfrak{c} = \dim_\C V\git G$,
and the tangent spaces to the orbits coincide on a Zariski open subset of $\mathfrak{c}$
(identifying each tangent space with a subspace of $V$). Let $N_G(\mathfrak{c})$ denote the subgroup of $G$ that fixes the set $\mathfrak{c}$
and let  $Z_G(\mathfrak{c})$ denote the subgroup of $G$ that fixes each point in
$\mathfrak{c}$.  By \cite[Lemma 2.7 and Theorem 2.9]{DadokKac}, the group
$\Gamma := N_G(\mathfrak{c})/Z_G(\mathfrak{c})$ is finite and the restriction
to $\mathfrak{c}$ induces an isomorphism between the complex
GIT quotients $V\git G$ and $\mathfrak{c}/\Gamma$.
If $\dim V\git G=1$, then $V$ is automatically polar.

\begin{proposition}
\label{prop:SU2R3Orbifold}
The symplectic quotient associated to $R_3$ is $\Z$-graded regularly symplectomorphic
to the orbifold $\C/\Z_4 = \C/\langle i\rangle$.
\end{proposition}
\begin{proof}
Let $V = R_3$ and $G=\SL_2$.  Then it is classical that $\C[V]^G=\C[f]$ where $f$ is homogeneous of degree 4. Thus
$V$ is polar, $\mathfrak{c}$ has dimension one and $\Gamma\simeq \Z/4\Z$. Using the usual weight basis $v_3$, $v_1$, $v_{-1}$ and $v_{-3}$ on $V$ we may choose
$\mathfrak{c}$ to be the span of $v:=v_3+v_{-3}$.  Then one calculates directly that $\mathfrak g(v)$ is perpendicular to $v$ relative to the hermitian form on $V$. Hence $v\in M$ and $\mathfrak{c}=\C\cdot v\subset M$. (One could just choose $v$ to be any nonzero point of $M$.)\
Since every closed $G$-orbit intersects $\mathfrak{c}$, we must have that $M=K\cdot \mathfrak{c}$.
Hence $M_0=M/K\simeq V\git G\simeq \mathfrak{c}/\Gamma$ and $\R[M]^K$ injects into $\R[\mathfrak{c}]^\Gamma$. But $\R[\mathfrak{c}]^\Gamma\otimes_\R\C\simeq\C[\mathfrak{c}\oplus \mathfrak{c}^\ast]^\Gamma$ is generated by invariants bihomogeneous of degrees $(4,0)$, $(0,4)$ and $(1,1)$ which obviously have lifts to $\R[V]^K$ since
$V\otimes_\R\C\simeq V\oplus V$.
Hence  $\R[\mathfrak{c}]^\Gamma$ lifts to  $\R[M]^K$  and $\mathfrak{c}\to M$ induces a
$\Z$-graded regular diffeomorphism of $\mathfrak{c}/\Gamma$ and $M_0=M/K$.

Now $\langle v,v\rangle\neq 0$ where $\langle\, ,\rangle$ is the hermitian form on $V$. Moreover, if $f\in\R[M]^K$, then $df(0)=0$ and  $df(v)$ annihilates $\mathfrak k\cdot v$ which is the perpendicular subspace to $\mathfrak{c}$ in $T_vM$. One has the analogous result at any nonzero point of $\mathfrak{c}$. Hence $\mathfrak{c}\to M$ induces a
$\Z$-graded regular symplectomorphism of $\mathfrak{c}/\Gamma$ and $M_0$.
\end{proof}

The symplectic
quotient  associated to $R_1$ is a point, and the quotient  associated to $2R_1$ is one-dimensional, hence polar, and we can use the arguments above
to show that it is $\Z$-graded regularly symplectomorphic to $\C/(\pm 1)$.
The case of $2R_1$ is also handled in
\cite[Examples 7.6 and   7.13]{ArmsGotayJennings}.  Then as the other representations
listed in Equation \eqref{eq:SU2candidates} are treated in Propositions \ref{prop:SU2NonOrbifolds},
\ref{prop:SU2R4Orbifold},   \ref{prop:SU2R3Orbifold} and Remark \ref{rem:R_2} we have completed the proof of
Theorem \ref{thrm:SU2}.

\begin{remark}
We note that the algebra isomorphisms inducing the $\Z$-graded regular symplectomorphisms
described in Propositions \ref{prop:SU2R4Orbifold} and
\ref{prop:SU2R3Orbifold} can be computed explicitly, and were suggested by the observation
that the Hilbert series coincide in \cite{HerbigSeaton}.  In particular, the algorithm
of Bedratyuk \cite{BedratyukSL2} yields the complex $\SL_2$-invariants.  Taking real
and imaginary parts and eliminating using the moment map, one can compute generators for
$\R[M_0]$ in each case.  For $R_3$, this yields one generator of degree $2$ and two of
degree $4$; for $R_4$, this yields three generators of degree $2$ and four of degree $3$.
Generating sets of the same degrees for the corresponding orbifolds can be computed using
standard techniques, and simple computations involving the Poisson brackets can be used to
determine the isomorphisms in terms of the generators.  That the resulting maps are homeomorphisms
of the corresponding (semialgebraic) orbit spaces can determined by comparing the inequalities
(or, more simply, the symmetric matrices determining the inequalities), computed using the methods
of \cite{ProcesiSchwarz}.
\end{remark}

% xxxxxxxxxxxxxxxxxxxxxxxxxxxxxxxxxxxxxxxxxxxxxxxxxxxxxxxxxxxxxxxxxxxxxxxxx
% xxxxxxxxxxxxxxxxxxxxxxxxxxxxxxxxxxxxxxxxxxxxxxxxxxxxxxxxxxxxxxxxxxxxxxxxx
% xxxxxxxxxxxxxxxxxxxxxxxxxxxxxxxxxxxxxxxxxxxxxxxxxxxxxxxxxxxxxxxxxxxxxxxxx

\bibliographystyle{amsplain}
\bibliography{HSS}

\end{document}